\documentclass[11pt,a4paper,reqno]{amsart}

\usepackage{amsmath,amssymb,amsfonts,amsthm,mathtools}
\usepackage[margin=1.15in]{geometry}
\usepackage{enumitem}
\usepackage[colorlinks=true,linkcolor=blue,citecolor=blue,urlcolor=blue]{hyperref}
\usepackage{microtype}

\numberwithin{equation}{section}

\newtheorem{theorem}{Theorem}[section]
\newtheorem{proposition}[theorem]{Proposition}
\newtheorem{lemma}[theorem]{Lemma}
\newtheorem{corollary}[theorem]{Corollary}
\newtheorem{assumption}[theorem]{Assumption}
\theoremstyle{definition}
\newtheorem{definition}[theorem]{Definition}
\newtheorem{example}[theorem]{Example}
\newtheorem{remark}[theorem]{Remark}

\newcommand{\Td}{\mathbb T^d}
\newcommand{\Tone}{\mathbb T}
\newcommand{\R}{\mathbb R}
\newcommand{\E}{\mathcal E}
\newcommand{\M}{\mathcal M}
\newcommand{\Ovac}{\mathcal O_m}
\newcommand{\Res}{\mathcal R_m}
\newcommand{\Pos}{\mathcal P_m}
\newcommand{\argmin}{\operatorname*{argmin}}

\newcommand{\dx}{\,dx}
\newcommand{\ResZero}{\mathcal R_m^{0}}
\newcommand{\ResGap}{\mathcal R_m^{\rm gap}}

\title[Remarks on selection problems in first-order MFGs]{Remarks on selection problems for first-order discounted Mean Field Games}
\author{Kengo Terai}
\address{Accenture Japan Ltd, Akasaka Intercity, 1-11-44 Akasaka, Minato-ku, Tokyo 107-8672}
\email{kengo.terai@accenture.com}


\begin{document}

\begin{abstract}

We study selection problems for the vanishing discount limit
of first-order stationary mean field games with local coupling.
The associated ergodic problem may admit several value functions
for the same density and ergodic constant.
We decompose the state space into the positive-density region, the essential
zero-density interior, and a residual set, and show that possible non-uniqueness of selected value functions can occur only
on the gap part of the residual set.
If the gradients of selected value functions coincide on this gap
residual set, then the selected value function is unique up to additive
constants; under compactness and stability assumptions, this yields
convergence of the whole normalized discounted family.
We show that the gap residual set is null for one-dimensional problems and
for a specific Hamiltonian in the multidimensional setting, and hence obtain convergence results in these cases.
\end{abstract}

\maketitle

\section{Introduction}

Mean field games were introduced by Lasry and Lions
\cite{LasryLions2007} and, independently, by Huang, Caines and Malham\'e
\cite{HuangCainesMalhame2006}, as a framework for differential games with a
large number of indistinguishable agents.  In the continuum limit, the
individual optimization problem and the evolution of the population
distribution are coupled through a Hamilton--Jacobi--Bellman equation and a
transport, Fokker--Planck, or continuity equation.  

In this paper, we study
stationary first-order mean field game systems with a local coupling
\begin{equation}\label{eq:discounted}
\begin{cases}
\varepsilon u^\varepsilon+H(x,Du^\varepsilon)=f(x,m^\varepsilon)
        &\qquad\text{in }\Td, \\[1mm]
\varepsilon m^\varepsilon
-\operatorname{div}\bigl(m^\varepsilon D_pH(x,Du^\varepsilon)\bigr)
=\varepsilon
        &\qquad\text{in }\Td, \\[1mm]
m^\varepsilon\ge0,\qquad
\displaystyle\int_{\Td}m^\varepsilon\dx=1.
\end{cases}
\end{equation}

Here $\varepsilon >0$, and $\Td$ denotes the $d$-dimensional flat torus, identified with $[0,1]^d$. The function $H: \Td \times \R^d \to \R$ is the Hamiltonian, and $f: \Td \times [0,\infty) \to \R$ is the local coupling. The unknowns are $u^\varepsilon : \Td \to \R$ and $m^\varepsilon : \Td \to [0,\infty)$.

The corresponding ergodic problem is
\begin{equation}\label{eq:ergodic-intro}
\begin{cases}
H(x,Du)=f(x,m)+\lambda &\qquad\text{in }\Td, \\[1mm]
-\operatorname{div}\bigl(mD_pH(x,Du)\bigr)=0 &\qquad\text{in }\Td, \\[1mm]
m\ge0,\qquad \displaystyle\int_{\Td}m\dx=1,
\end{cases}
\end{equation}
where \(u:\Td\to\R\) is the value function,
\(m:\Td\to[0,\infty)\) is the density of agents, and
\(\lambda\in\R\) is the ergodic constant.

Mathematical analysis of first-order mean field games with local coupling
differs substantially from that of uniformly elliptic second-order systems.
In the first-order case, the continuity equation contains no diffusive term,
and hence no elliptic regularization is available for the population density.
Consequently, \(m\) may have low regularity and may vanish on sets of positive
measure.  This lack of regularization is particularly relevant for local
couplings, since the Hamilton--Jacobi equation contains the pointwise term
\(f(x,m)\).  These considerations naturally lead to the notion of weak solutions in general.
Weak solution frameworks for first-order mean field games with local coupling
have been developed through dual variational formulations and monotonicity
methods; see
\cite{Cardaliaguet2015,CardaliaguetGraber2015,
Graber2014LinearlyBounded,FerreiraGomes2018}.
Regularity questions have also been studied in several settings; see, for
instance,
\cite{CardaliaguetPorrettaTonon2015,
GraberMeszaros2018,Munoz2022,AlharbiGomesDiFazioUcer2025}.
We also refer to \cite{MimikosStamatopoulosMunoz2024} for one-dimensional
regularity and long-time behavior, and to
\cite{CardaliaguetMunozPorretta2024} for support propagation and free
boundary regularity.

As a motivation for the present work, we recall the vanishing discount
problem for scalar Hamilton--Jacobi equations.  Let
\(G:\Td\times\R^d\to\R\) be a coercive Hamiltonian.  It is well known that,
for a suitable constant \(\bar c\), the ergodic problem
\begin{equation}\label{eq:ergodic-HJ-intro}
        G(x,Dv)=\bar c
        \qquad\text{in }\Td
\end{equation}
admits viscosity solutions; see the seminal work
\cite{LionsPapanicolaouVaradhan1987}.  This equation appears naturally in
several asymptotic problems for Hamilton--Jacobi equations, including
homogenization and large-time behavior.
A standard approximation of \eqref{eq:ergodic-HJ-intro} is the discounted
problem
\begin{equation}\label{eq:discounted-HJ-intro}
        \varepsilon v^\varepsilon+G(x,Dv^\varepsilon)=0
        \qquad\text{in }\Td .
\end{equation}
This equation arises in optimal control and differential games, where
\(\varepsilon>0\) represents a discount factor.  The term
\(\varepsilon v^\varepsilon\) provides the monotonicity in the unknown
function needed for the comparison principle, and hence the discounted
problem has a unique viscosity solution.  Moreover, coercivity gives
Lipschitz estimates for suitable normalizations.  By the Arzel\`a--Ascoli
theorem, along a subsequence,
\[
        \varepsilon v^\varepsilon\to -\bar c,
        \qquad
        v^\varepsilon-\min_{\Td}v^\varepsilon\to v
\]
locally uniformly as \(\varepsilon\to0\).  By the stability of viscosity
solutions, the limit \(v\) solves \eqref{eq:ergodic-HJ-intro}.

On the other hand, since \eqref{eq:ergodic-HJ-intro} may have multiple solutions in general,
subsequential convergence does not by itself identify a unique limit.  The
selection problem is to prove convergence of the whole normalized family and
to characterize the solution selected by the discount approximation.
For scalar Hamilton--Jacobi equations, this problem has been studied by weak
KAM theory, nonlinear adjoint methods, and viscosity Mather measures; see, for instance,
\cite{DaviniFathiIturriagaZavidovique2016,MitakeTran2017,
IshiiMitakeTran2017} and the references therein. Outside the standard convex setting, convergence along the whole family may
fail; for example, \cite{Ziliotto2019} constructed a counterexample for a
nonconvex Hamiltonian.

The vanishing discount problem for first-order mean field games with local
coupling was studied in
\cite{GomesMitakeTerai2020,MitakeTerai2023}.  In the mechanical Hamiltonian
case, \cite{GomesMitakeTerai2020} proved convergence under a
small-oscillation condition on the potential.  In that setting, the corresponding ergodic problem has a unique classical solution up to additive constants; hence the limiting value function is determined by uniqueness rather than by an additional selection mechanism.  The same work then turned to weak solutions introduced in
\cite{FerreiraGomes2018}, and showed that the ergodic problem may admit
several weak solutions \(u\) for the same density \(m\) and the same constant
\(\lambda\).  This led to a selection problem, and a selection criterion for
subsequential discounted limits was derived.  The work
\cite{MitakeTerai2023} developed the vanishing discount analysis in a
different weak solution framework, related to the variational formulations in
\cite{Cardaliaguet2015,CardaliaguetGraber2015,Graber2014LinearlyBounded}.
It established existence and uniqueness of discounted weak solutions,
stability toward the ergodic problem, examples of non-uniqueness for weak solutions to the
ergodic problem, and a necessary selection condition for subsequential
limits.

Related asymptotic problems have also been studied recently.  For instance,
\cite{IturriagaMendicoWangXu2025} studied time discretization and vanishing
discount problems for first-order discounted mean field game systems with
nonlocal coupling, using methods from weak KAM theory.
The work \cite{CaiQiSuTan2025} studied
one-dimensional stationary mean field games with vanishing potential.

The present paper starts from the selection viewpoint developed in
\cite{GomesMitakeTerai2020,MitakeTerai2023}.  We take as input the selection
condition derived there, and study when the corresponding selection
functional has a unique minimizer up to additive constants.  We show that
possible non-uniqueness of selected value functions is localized on the gap
part of the residual set.  If the gradients of selected value functions
coincide a.e. on this set, then subsequential convergence is upgraded to
convergence along the whole normalized discounted family.

\subsection{Preliminaries and setting}

We introduce the notation and weak solution framework used in the selection
argument.  We do not aim at the most general formulation; rather, we fix an
ergodic density \(m\) and an ergodic constant \(\lambda\), and study the
admissible value functions associated with this pair.

We first state the definition of weak solutions used in this paper. This is based on the work in \cite{Cardaliaguet2015,CardaliaguetGraber2015,Graber2014LinearlyBounded}. Throughout the paper, we fix \(s>1\) and \(q>1\), and denote by \(s'\) and \(q'\) their conjugate
exponents.  We also write \(p:=q'\) when comparing with the standard
variational notation. To define the selection functional later, we assume that
\begin{equation}\label{exponent-relation}
        W^{1,s}(\Td)\hookrightarrow L^{q'}(\Td).
\end{equation}

\begin{definition}[Weak solutions]\label{def:E}
A triple $(u,m,\lambda) \in W^{1,s}(\Td) \times L^q(\Td) \times \R$ is a weak solution to \eqref{eq:ergodic-intro} if 
\begin{enumerate}[label=(E\arabic*)]

\item $m \ge 0$ a.e., $\int_{\Td}m\dx=1$ and
\begin{equation}\label{integrability-flux}
        mD_pH(x,Du)\in L^{s'}(\Td;\R^d),
\end{equation}

\item The Hamilton--Jacobi inequality holds a.e. on \(\Td\):
\begin{equation}\label{eq:HJ-ineq}
        H(x,Du)\le f(x,m)+\lambda,
\end{equation}
\item The Hamilton--Jacobi equation holds a.e. on the positive-density region:
\begin{equation}\label{eq:HJ-eq-P}
        H(x,Du)=f(x,m)+\lambda
        \quad\text{a.e. on }\{m>0\},
\end{equation}
\item The continuity equation holds in distributions:
\begin{equation}\label{eq:FPweak-ep}
        -\operatorname{div}\bigl(mD_pH(x,Du)\bigr)=0
        \quad\text{in }\mathcal D'(\Td).
\end{equation}
\end{enumerate}
We denote $u \in \E(m,\lambda)$ when $(u,m,\lambda)$ is a weak solution to \eqref{eq:ergodic-intro} for fixed $(m,\lambda)$.
\end{definition}

\begin{remark}
The integrability requirement \eqref{integrability-flux}
 allows the continuity equation to be tested against \(W^{1,s}\)-functions.  
In the framework of weak solutions introduced by \cite{Cardaliaguet2015,CardaliaguetGraber2015,Graber2014LinearlyBounded}, this condition
is built into the choice of exponents.  For Hamiltonians $H(x,\cdot)$ with \(r\)-growth and
couplings yielding \(m\in L^q\), one has \(u\in W^{1,pr}\) and
\[
        mD_pH(x,Du)\in L^{(pr)'}(\Td;\R^d).
\]
\end{remark}

For the discounted problem, we use the natural analogue of Definition
\ref{def:E}, with the equations modified according to \eqref{eq:discounted};
well-posedness of the discounted problem in this weak-solution framework is
established in
\cite[Theorem 1.2]{MitakeTerai2023}.

We now introduce the selection functional.  This functional appears as a
necessary condition for vanishing discount limits: in the work of
\cite{GomesMitakeTerai2020,MitakeTerai2023}, every subsequential limit of the
normalized discounted value functions minimizes this functional among
$\E(m,\lambda)$ where \((m,\lambda)\) is the limiting density and the
ergodic constant; see Appendix \ref{sec:selection}.

\begin{definition}[Selection functional and selected minimizers]\label{def:Jmin}
For \(u\in\E(m,\lambda)\), define
\begin{equation*}
        J(u):=\int_{\Td}\langle u\rangle m\dx,
        \qquad
        \langle u\rangle:=u-\int_{\Td}u\dx.
\end{equation*}
The set of selected minimizers is
\begin{equation*}
        \M_J:=\argmin_{u\in\E(m,\lambda)}J(u).
\end{equation*}
\end{definition}

The functional \(J\) is well-defined by \eqref{exponent-relation} and
\(m\in L^q(\Td)\).  Since \(|\Td|=1\) and
\(\int_{\Td}m\dx=1\), we may also write
\begin{equation}\label{eq:J-linear-form}
        J(u)=\int_{\Td}u(x)(m(x)-1)\dx .
\end{equation}
In particular, \(J\) is invariant under the addition of constants to \(u\).
The main problem of this paper is to prove uniqueness of
this minimizer under structural conditions.

We next introduce the three regions that will be used in the uniqueness proof. Since \(m\in L^1\) is an equivalence class, the pointwise set \(\{m=0\}\) depends on the representative.  We therefore define the zero-density interior intrinsically.

\begin{definition}[Essential zero-density interior]\label{def:essential-zero-density}
The essential zero-density interior of \(m\) is
\begin{equation}\label{eq:essential-zero-density}
        \Ovac:=
        \bigcup\left\{
        U\subset\Td:\ U\text{ is open and }m=0\text{ a.e. in }U
        \right\}.
\end{equation}
After fixing a measurable representative of \(m\), we set
\begin{equation*}
        \Pos:=\{m>0\}\setminus\Ovac,
        \qquad
        \Res:=\Td\setminus(\Pos\cup\Ovac).
\end{equation*}
\end{definition}

The set \(\Ovac\) is open by definition and depends only on the \(L^1\)-equivalence class of \(m\), since the condition \(m=0\) a.e. in an open set is representative-independent.  Moreover, because \(\Td\) is second countable, the union in \eqref{eq:essential-zero-density} can be reduced to a countable subunion; hence \(m=0\) a.e. in \(\Ovac\).  The sets \(\Pos\) and \(\Res\) may depend on the chosen representative of \(m\), but only up to Lebesgue null sets.  All statements involving \(\Pos\) and \(\Res\) will therefore be understood in the a.e. sense.  With this convention, we have the measurable decomposition
\[
        \Td=\Ovac\dot\cup\Pos\dot\cup\Res
\]
up to null sets.

\subsection{Contributions and main results}

We now state the structural criterion which is the main result of this paper.
The proof is based on a decomposition of the state space into the
positive-density region, the essential zero-density interior, and a residual
set.  The positive-density region is treated by the monotonicity argument,
the zero-density interior by a local bump construction used in Perron's method, and the
residual set by separating the threshold part from the gap part.

Throughout the paper, we assume the following structural conditions on the
Hamiltonian \(H\) and the coupling \(f\).

\begin{assumption}[Hamiltonian]\label{ass:H}
The Hamiltonian \(H:\Td\times\R^d\to\R\) satisfies the following conditions.
\begin{enumerate}[label=(H\arabic*)]
\item \(H\in C^1(\Td\times\R^d)\).
\item For every \(x\in\Td\), the map \(p\mapsto H(x,p)\) is strictly convex.
\item \(H\) is coercive in \(p\), uniformly in \(x\):
\[
        \lim_{|p|\to\infty}\inf_{x\in\Td}H(x,p)=+\infty.
\]
\end{enumerate}
\end{assumption}

\begin{assumption}[Coupling]\label{ass:f}
The coupling \(f:\Td\times[0,\infty)\to\R\) satisfies the following
conditions.
\begin{enumerate}[label=(F\arabic*)]
\item \(f\) is continuous.
\item For every \(x\in\Td\), the map \(m\mapsto f(x,m)\) is strictly
increasing.
\item There exists \(C>0\) such that
\[
        |f(x,r)|\le C(1+r^{q-1})
        \qquad
        \text{for all }(x,r)\in\Td\times[0,\infty).
\]
\end{enumerate}
\end{assumption}

We use the following compactness and stability input for the vanishing
discount approximation.  Sufficient conditions for this input are given in
\cite[Theorem 1.3]{MitakeTerai2023}.

\begin{assumption}[Compactness and stability]\label{ass:discounted-compactness}
Let \((u^\varepsilon,m^\varepsilon) \in W^{1,s}(\Td)\times L^q(\Td)\) be a family of weak solutions to
\eqref{eq:discounted}.  For every sequence \(\varepsilon_n\to 0\), there
exist a subsequence, not relabeled, and a triple
\((\bar u,m,\lambda)\in W^{1,s}(\Td)\times L^q(\Td)\times\R\) such that
\[
    \langle u^{\varepsilon_n}\rangle
    \rightharpoonup \bar u
    \quad\text{weakly in } W^{1,s}(\Td),
\]
\[
    m^{\varepsilon_n}
    \rightharpoonup m
    \quad\text{weakly in } L^q(\Td),
\]
and
\[
    \varepsilon_n\int_{\Td}u^{\varepsilon_n}\dx
    \to -\lambda.
\]
We also assume that \((\bar u,m,\lambda)\) is a weak solution to
\eqref{eq:ergodic-intro}.
\end{assumption}

Before stating the residual gradient condition, we refine the residual
decomposition.  Set
\begin{equation*}
        \ell(x):=\lambda+f(x,0),
        \qquad
        K_x:=\{p\in\R^d:\ H(x,p)\le \ell(x)\},
\end{equation*}
and
\begin{equation*}
        \underline H(x):=\min_{p\in\R^d}H(x,p).
\end{equation*}
By the continuity of \(H\) and the uniform coercivity in \(p\),
\(\underline H\) is continuous on \(\Td\).  Moreover, since \(H(x,\cdot)\) is
strictly convex, \(K_x\) is a singleton exactly when
\[
        \ell(x)=\underline H(x).
\]
If \(\ell(x)<\underline H(x)\), then \(K_x=\emptyset\); if
\(\ell(x)>\underline H(x)\), then \(K_x\) contains more than one point.

Since \(m=0\) a.e. on \(\Res\), every \(u\in\E(m,\lambda)\) satisfies
\[
        H(x,Du)\le \ell(x)
        \quad\text{for a.e. }x\in\Res.
\]
Thus \(\Res\cap\{\ell<\underline H\}\) is null whenever
\(\E(m,\lambda)\neq\emptyset\).  Hence, up to null sets, we can write
\begin{equation*}
        \ResZero:=\Res\cap\{\ell=\underline H\},
        \qquad
        \ResGap:=\Res\cap\{\ell>\underline H\}.
\end{equation*}

On \(\ResZero\), no additional gradient condition is needed: the inequality
\(H(x,Du)\le \ell(x)=\underline H(x)\) forces \(Du\) to be the unique minimizer
of \(p\mapsto H(x,p)\).  Thus every admissible value function has the same
gradient a.e. on \(\ResZero\).

\begin{assumption}[Gradient uniqueness on the gap residual set]\label{ass:residual}
For any \(u_0,u_1\in\M_J\),
\begin{equation*}
        Du_0=Du_1\quad\text{a.e. on }\ResGap.
\end{equation*}
\end{assumption}

\begin{remark}
Assumption \ref{ass:residual} should be viewed as an additional sufficient
criterion.  We prove below that gradient uniqueness holds outside
\(\ResGap\).  Thus Assumption \ref{ass:residual} rules out the only remaining
possible obstruction.
\end{remark}

We now state the main result of this paper.

\begin{theorem}[Uniqueness criterion and convergence]\label{thm:main-unique}
Suppose that Assumptions \ref{ass:H}, \ref{ass:f}, and \ref{ass:residual} hold.
Then \(\M_J\) contains at most one element up to additive constants.

Consequently, if Assumption \ref{ass:discounted-compactness} holds, the normalized discounted value functions converge along the whole family:
\[
        \langle u^\varepsilon\rangle \rightharpoonup u_*
        \quad\text{weakly in }W^{1,s}(\Td),
\]
where \(u_*\) is the unique selected minimizer normalized by \(\int_{\Td}u_*\dx=0\).
\end{theorem}

Section \ref{sec:examples} applies the criterion to obtain concrete
convergence results for one-dimensional problems and for a specific
Hamiltonian in the multidimensional setting.

The paper is organized as follows.  Section \ref{sec:ergodic-uniqueness}
reviews uniqueness of the density and the ergodic constant and gives a simple
example of non-uniqueness of the value function.  Section
\ref{sec:grad-rigid} proves the abstract uniqueness criterion.  Section
\ref{sec:examples} gives concrete applications.

\section{Uniqueness and non-uniqueness in the ergodic problem}\label{sec:ergodic-uniqueness}

We begin by recalling the standard uniqueness property of the density and the
ergodic constant.  This is a basic consequence of the monotonicity structure
in mean field games, and it explains why it is natural to fix \((m,\lambda)\)
and study the possible multiplicity of the value function \(u\).

Such a uniqueness property is also encoded in the variational formulation of
first-order mean field games; see Cardaliaguet and Graber
\cite{CardaliaguetGraber2015}. For completeness, we give a direct proof based on the Lasry--Lions monotonicity method for weak solutions satisfying
Definition \ref{def:E}.

\begin{proposition}
\label{prop:unique-density-lambda}
Suppose that Assumptions \ref{ass:H} and \ref{ass:f} hold.  Let
\[
        (u_i,m_i,\lambda_i),\qquad i=0,1,
\]
be two weak solutions to \eqref{eq:ergodic-intro}.  Then
\[
        m_0=m_1\quad\text{a.e. in }\Td,
        \qquad
        \lambda_0=\lambda_1.
\]
Moreover, for the set \(\Pos\) associated with this common density \(m\), we have
\begin{equation}\label{grad-unique-pos-density}
        Du_0=Du_1
        \quad\text{a.e. on }\Pos.
\end{equation}
\end{proposition}

\begin{proof}
Set
\[
        p_i:=Du_i,
        \qquad
        f_i:=f(x,m_i),
        \qquad i=0,1.
\]

We first prove uniqueness of the density.  On \(\{m_0>0\}\), the
Hamilton--Jacobi equality \eqref{eq:HJ-eq-P} for \(u_0\) and the Hamilton--Jacobi inequality \eqref{eq:HJ-ineq} for
\(u_1\) give
\[
        H(x,p_0)-H(x,p_1)
        \ge
        f_0-f_1+\lambda_0-\lambda_1.
\]
On the other hand, by convexity of \(H(x,\cdot)\),
\[
        H(x,p_1)-H(x,p_0)
        \ge
        D_pH(x,p_0)\cdot(p_1-p_0),
\]
and hence
\[
        f_0-f_1+\lambda_0-\lambda_1
        \le
        -D_pH(x,p_0)\cdot(p_1-p_0).
\]
Multiplying by \(m_0\) and using the continuity equation \eqref{eq:FPweak-ep} for \(u_0\) tested
by \(u_1-u_0\), we obtain
\[
        \int_{\Td}m_0(f_0-f_1)\dx+\lambda_0-\lambda_1
        \le0.
\]

Similarly,  we get
\[
        \int_{\Td}m_1(f_1-f_0)\dx+\lambda_1-\lambda_0
        \le0.
\]
Adding the two inequalities yields
\[
        \int_{\Td}
        \bigl(f(x,m_0)-f(x,m_1)\bigr)(m_0-m_1)\dx
        \le0.
\]
By the monotonicity of \(f\) in the second variable, the integrand is
nonnegative.  Therefore it must vanish a.e. in $\Td$.  Since \(f(x,\cdot)\) is strictly
increasing, we conclude that
\[
        m_0=m_1
        \quad\text{a.e. in }\Td.
\]

Set \(m:=m_0=m_1\).  We next prove \eqref{grad-unique-pos-density}.  Using the continuity equations \eqref{eq:FPweak-ep} for \(u_0\) and \(u_1\), tested
by \(u_1-u_0\) and \(u_0-u_1\), respectively, we get
\[
        \int_{\Td}
        mD_pH(x,p_0)\cdot(p_1-p_0)\dx=0,
\]
and
\[
        \int_{\Td}
        mD_pH(x,p_1)\cdot(p_0-p_1)\dx=0.
\]
Adding these identities gives
\[
        \int_{\Td}
        m\bigl(D_pH(x,p_0)-D_pH(x,p_1)\bigr)\cdot(p_0-p_1)\dx=0.
\]

By the strict convexity of \(H(x,\cdot)\), we have
\[
        p_0=p_1
        \quad\text{a.e. on }\{m>0\}.
\]
In particular, we have \eqref{grad-unique-pos-density}.

Finally, since \(m\ge0\) and \(\int_{\Td}m\dx=1\), the set \(\{m>0\}\) has
positive measure.  On this set, the Hamilton--Jacobi equality \eqref{eq:HJ-eq-P} gives
\[
        H(x,Du_i)=f(x,m)+\lambda_i,
        \qquad i=0,1.
\]
Since \(Du_0=Du_1\) a.e. on \(\{m>0\}\), we obtain $\lambda_0=\lambda_1$.
\end{proof}

We next recall a simple example of multiple value functions in $\{m=0\}$. 

\begin{example}
\label{ex:multiplicity-u}
Let \(d=1\), and consider
\[
        H(x,p)=\frac12|p|^2+V(x),
        \qquad
        f(x,m)=m.
\]
For instance, take
\[
        V(x)=\pi\cos(2\pi x),
        \qquad
        \lambda=0.
\]
Then
\[
        m(x):=(V(x)-\lambda)_+
        =
        \pi(\cos(2\pi x))_+
\]
satisfies
\[
        \int_{\Tone}m(x)\dx=1,
\]
and the open set
\[
        O:=\{x\in\Tone:\ V(x)<\lambda\}
\]
is nonempty.

Then \(u\equiv0\) belongs to \(\E(m,\lambda)\).  Indeed, on
\(\{m>0\}=\{V>\lambda\}\), one has
\[
        H(x,0)=V(x)=m(x)+\lambda,
\]
while on \(\{m=0\}=\{V\le\lambda\}\),
\[
        H(x,0)=V(x)\le\lambda=f(x,0)+\lambda.
\]
The continuity equation is satisfied since \(Du=0\).

The value function, however, is not unique.  Take a nonzero
\(\phi\in C_c^1(O)\), and set \(\psi=\delta\phi\), where \(\delta>0\) is
small enough that
\[
        \frac12|\psi_x|^2+V\le\lambda
        \quad\text{a.e. in }O.
\]
Extending \(\psi\) by zero outside \(O\), we obtain another element of
\(\E(m,\lambda)\).  Indeed, the Hamilton--Jacobi equality is unchanged on
\(\{m>0\}\), the Hamilton--Jacobi inequality holds on \(\{m=0\}\), and the
continuity equation is unchanged because the modification is supported where
\(m=0\).

Thus the same density \(m\) and the same ergodic constant \(\lambda\) may
admit multiple value functions $u$.  This illustrates why the vanishing
discount problem leads to a selection problem for the value function.
\end{example}

\section{Uniqueness of the gradient for selected minimizers}\label{sec:grad-rigid}
In the previous section, we proved uniqueness of the gradient on the
positive-density region \(\Pos\).  We now prove the corresponding result on
the essential zero-density interior \(\Ovac\), and then prove the main
theorem.

\subsection{Uniqueness of the gradient on the zero-density interior}
\label{sec:bump}

We first prepare the following statement. 
\begin{lemma}\label{lem:local-admissible}
Suppose that Assumptions \ref{ass:H} and \ref{ass:f} hold.
Let \(u\in\E(m,\lambda)\), and let \(B\Subset\Ovac\) be an open ball.  Let \(\tilde u\in W^{1,s}(\Td)\) satisfy
\[
        \tilde u=u\quad\text{a.e. on }\Td\setminus B
\]
and suppose that \(\tilde u\) is locally Lipschitz in \(B\) and is a viscosity subsolution of
\begin{equation*}
        H(x,Dv)=\ell(x)\quad\text{in }B.
\end{equation*}
 Then \(\tilde u\in\E(m,\lambda)\).
\end{lemma}

\begin{proof}
Since \(B\Subset\Ovac\), one has \(m=0\) a.e. in \(B\).  Also, because $D\tilde u=Du$ a.e. on $\Td\setminus B$,
we have
\[
        mD_pH(x,D\tilde u)=mD_pH(x,Du)
        \quad\text{a.e. on }\Td .
\]
Therefore \eqref{integrability-flux} and 
\eqref{eq:FPweak-ep} for \(\tilde u\) follow from those for \(u\).  

It remains to verify the Hamilton--Jacobi conditions.  Since \(\tilde u\) is
locally Lipschitz and a viscosity subsolution in \(B\), we have
\[
        H(x,D\tilde u)\le \ell(x)
        \quad\text{for a.e. }x\in B.
\]
On \(B\), \(m=0\) a.e.; hence
\[
        H(x,D\tilde u)
        \le \lambda+f(x,0)
        =
        \lambda+f(x,m)
        \quad\text{a.e. in }B.
\]
On \(\Td\setminus B\), we have \(D\tilde u=Du\) a.e., and therefore
\eqref{eq:HJ-ineq} for \(\tilde u\) follows from \eqref{eq:HJ-ineq} for
\(u\). 
Finally, since \(B\cap\{m>0\}\) is null and \(D\tilde u=Du\) a.e. on
\(\Td\setminus B\), the equality on the positive-density region
\eqref{eq:HJ-eq-P} follows from \eqref{eq:HJ-eq-P} for \(u\).
\end{proof}

We now prove that selected minimizers are viscosity solutions in the zero-density interior. The proof below is a local version of the bump construction used in Perron's method for viscosity solutions; see \cite{Ishii1987Perron}.  In the standard Perron argument, maximality is provided by the envelope construction.  Here it is provided by the minimality of the selection functional.
\begin{lemma}\label{lem:bump-super}
Suppose that Assumptions \ref{ass:H} and \ref{ass:f} hold. Let \(u\in\M_J\).  Then \(u\) is locally Lipschitz continuous on \(\Ovac\) and a viscosity solution to
\begin{equation}\label{eq:vac-HJ2}
        H(x,Du)=\ell(x):=\lambda+f(x,0)
        \quad\text{in }\Ovac.
\end{equation}
\end{lemma}

\begin{proof}
We first prove that \(u\) is locally Lipschitz continuous on \(\Ovac\) and a viscosity subsolution in \(\Ovac\).
Because \(m=0\) a.e. in \(\Ovac\), the Hamilton--Jacobi inequality in the
definition of \(\E(m,\lambda)\) gives
\[
        H(x,Du)\le \lambda+f(x,0)=\ell(x)
        \quad\text{a.e. in }\Ovac.
\]
By the coercivity of \(H\) in \(p\), this inequality implies
\[
        u\in W^{1,\infty}_{\mathrm{loc}}(\Ovac).
\]
By the convexity of \(H\) in \(p\), and by the standard equivalence between
a.e. subsolutions and viscosity subsolutions for convex first-order
Hamilton--Jacobi equations, \(u\) is a viscosity subsolution of
\eqref{eq:vac-HJ2}; see for instance, \cite[Theorem 2.32]{Tran2021HJ}.

We next prove the supersolution property.  Suppose by contradiction that
\(u\) is not a viscosity supersolution at some \(x_0\in\Ovac\).  Then there
exists \(\phi\in C^1(\Ovac)\) such that \(u-\phi\) has a local minimum at
\(x_0\) and
\[
        H(x_0,D\phi(x_0))<\ell(x_0).
\]
Replacing \(\phi\) by
\[
        \phi(x)-\gamma |x-x_0|^2
\]
with \(\gamma>0\) sufficiently small, we may assume that the local minimum is
strict, without changing the gradient at \(x_0\).

Normalize the test function by setting
\[
        \psi(x):=\phi(x)+u(x_0)-\phi(x_0).
\]
Then \(u(x_0)=\psi(x_0)\), the function \(u-\psi\) has a strict local minimum
equal to \(0\) at \(x_0\), and
\[
        H(x_0,D\psi(x_0))<\ell(x_0).
\]
By continuity, there exist a ball \(B\Subset\Ovac\), centered at \(x_0\), and
\(\eta>0\) such that
\[
        H(x,D\psi(x))\le \ell(x)-\eta
        \quad\text{in }B.
\]
Taking \(B\) smaller if necessary, the strictness of the minimum gives
\[
        u-\psi\ge 2\alpha
        \quad\text{on }\partial B
\]
for some \(\alpha>0\).  Since \(u-\psi\) is continuous on \(\overline B\),
there exists \(\rho>0\) such that
\[
        u-\psi\ge \alpha
        \quad\text{whenever }x\in B
        \text{ and } \operatorname{dist}(x,\partial B)<\rho .
\]
Choose \(0<\delta<\alpha\), and define
\[
        \tilde u(x)=
        \begin{cases}
        \max\{u(x),\psi(x)+\delta\},&x\in B,\\
        u(x),&x\notin B.
        \end{cases}
\]
On the set
\[
        \{x\in B:\operatorname{dist}(x,\partial B)<\rho\},
\]
we have \(\psi+\delta<u\), and hence \(\tilde u=u\) in a neighborhood of
\(\partial B\).  Therefore the piecewise definition is compatible across
\(\partial B\).

Moreover, since \(u\in W^{1,\infty}_{\mathrm{loc}}(\Ovac)\) and
\(\psi\in C^1(B)\), the Sobolev chain rule and the identity
\[
        \max\{u,\psi+\delta\}
        =
        \frac{u+(\psi+\delta)+|u-(\psi+\delta)|}{2}
\]
imply
\[
        \max\{u,\psi+\delta\}\in W^{1,\infty}_{\mathrm{loc}}(B).
\]
Because \(\tilde u=u\) in a neighborhood of \(\partial B\), we obtain
\[
        \tilde u\in W^{1,s}(\Td).
\]
Also, near \(x_0\), one has
\[
        \psi(x_0)+\delta=u(x_0)+\delta>u(x_0),
\]
so \(\tilde u>u\) on a set of positive measure.

By the standard stability of viscosity subsolutions under finite maxima,
\(\tilde u\) is a viscosity subsolution of
\[
        H(x,Dv)=\ell(x)
        \quad\text{in }B.
\]
By Lemma \ref{lem:local-admissible}, we have $\tilde u\in\E(m,\lambda)$.

Finally, using \eqref{eq:J-linear-form} and the fact that \(m=0\) a.e. in
\(B\), we obtain
\[
        J(\tilde u)-J(u)
        =
        \int_{\Td}(\tilde u-u)(m-1)\dx
        =
        -\int_B(\tilde u-u)\dx<0.
\]
This contradicts the minimality of \(u\).  Hence \(u\) is a viscosity
supersolution in \(\Ovac\).
\end{proof}

These properties imply the following uniqueness result for the gradient on
\(\Ovac\).
\begin{proposition}\label{prop:ae-vac}
Suppose that Assumptions \ref{ass:H} and \ref{ass:f} hold. Let \(u_0,u_1\in\M_J\).  Then,
\begin{equation}\label{eq:proof-O}
        Du_0=Du_1\quad\text{a.e. on }\Ovac.
\end{equation}
\end{proposition}

\begin{proof}
Let \(u_0,u_1\in\M_J\), and set
\[
        w:=\frac{u_0+u_1}{2}.
\]
We first prove that \(w\in\M_J\).

Since \(u_0,u_1\in W^{1,s}(\Td)\), we have \(w\in W^{1,s}(\Td)\).  
By Proposition \ref{prop:unique-density-lambda}, applied to the two weak
solutions \((u_0,m,\lambda)\) and \((u_1,m,\lambda)\), we have
\[
        Du_0=Du_1
        \quad\text{a.e. on }\Pos.
\]
In particular,
\[
        Dw=Du_0=Du_1
        \quad\text{a.e. on }\Pos.
\]
On the other hand, by the definition of \(\Pos\), we have
\[
        m=0
        \quad\text{a.e. on }\Td\setminus\Pos,
\]
which yields
\begin{equation}\label{admissible-conti-w}
        mD_pH(x,Dw)=mD_pH(x,Du_0)
        \quad\text{a.e. on }\Td.
\end{equation}
Since \(u_0\in\E(m,\lambda)\), this gives
\[
        mD_pH(x,Dw)\in L^{s'}(\Td;\R^d),
\]
and the continuity equation for \(w\) follows from the continuity equation
for \(u_0\) and \eqref{admissible-conti-w}. Thus $w$ satisfies \eqref{integrability-flux} and \eqref{eq:FPweak-ep} in Definition \ref{def:E}.

Next, by the convexity of \(H(x,\cdot)\),
\[
        H(x,Dw)
        \le
        \frac12 H(x,Du_0)+\frac12 H(x,Du_1)
        \le
        f(x,m)+\lambda
        \quad\text{a.e. on }\Td.
\]
Thus we obtain \eqref{eq:HJ-ineq} for \(w\).  Moreover, on the
positive-density region, the equality is preserved.  Indeed,
 on \(\Pos\) we have
\(Dw=Du_0=Du_1\).  Hence
\[
        H(x,Dw)=H(x,Du_0)=f(x,m)+\lambda
        \quad\text{a.e. on }\{m>0\},
\]
which yields \eqref{eq:HJ-eq-P}. Therefore \(w\in\E(m,\lambda)\).

Finally, by the linearity of \(J\), one has
\[
        J(w)
        =
        \frac12J(u_0)+\frac12J(u_1)
        =
        \min_{\E(m,\lambda)}J.
\]
Thus \(w\in\M_J\).

By Lemma \ref{lem:bump-super}, the functions \(u_0,u_1\), and \(w\) are
locally Lipschitz viscosity solutions of
\[
        H(x,Du)=\lambda+f(x,0)
        \quad\text{in }\Ovac.
\]
Therefore, at a.e. differentiability point in \(\Ovac\),
\[
        H(x,Du_0)=H(x,Du_1)=H(x,Dw)=\lambda+f(x,0).
\]
Since \(Dw=(Du_0+Du_1)/2\), strict convexity of \(H(x,\cdot)\) implies \eqref{eq:proof-O}.
\end{proof}

\subsection{Proof of the main result} 

\begin{proof}[Proof of Theorem \ref{thm:main-unique}]
Let \(u_0,u_1\in\M_J\).  By Propositions \ref{prop:unique-density-lambda} and \ref{prop:ae-vac}, we have
\[
        Du_0=Du_1
        \quad\text{a.e. on }\Pos\cup\Ovac.
\]
By Assumption \ref{ass:residual}, the same identity holds a.e. on
\(\ResGap\).  On \(\ResZero\), the inequality
\(H(x,Du_i)\le \ell(x)=\underline H(x)\) forces both gradients to be the
unique minimizer of \(p\mapsto H(x,p)\).  Hence
\(Du_0=Du_1\) a.e. on \(\ResZero\).

  Since
\[
        \Td=\Pos\dot\cup\Ovac\dot\cup\ResZero\dot\cup\ResGap
\]
up to null sets, we obtain
\[
        Du_0=Du_1
        \quad\text{a.e. on }\Td.
\]
Since \(\Td\) is connected, \(u_1-u_0\) is constant a.e. on $\Td$.  Thus
\(\M_J\) contains at most one element up to additive constants.

Under Assumption \ref{ass:discounted-compactness}, Proposition
\ref{prop:selection} implies that every subsequential limit belongs to
\(\M_J\).  Since \(\M_J\) has a unique normalized element, the whole
normalized family converges.
\end{proof}

\section{Concrete convergence results via the gap residual condition}
\label{sec:examples}

In this section we turn the uniqueness criterion into concrete convergence
results.  The residual decomposition separates the zero-gap part
\(\ResZero\), where the Hamilton--Jacobi sublevel set is a singleton, from
the gap part \(\ResGap\), where a residual non-uniqueness of gradients could
in principle remain.  In the classes considered below, we prove that
\(\ResGap\) is null.  Thus Assumption \ref{ass:residual} imposes no
additional restriction, and, under Assumption
\ref{ass:discounted-compactness}, Theorem \ref{thm:main-unique} yields
convergence of the whole normalized discounted family.

\subsection{A threshold criterion}
\label{sec:threshold-criterion}

We first record a simple criterion which will be used below.

\begin{lemma}
\label{lem:threshold-no-gap}
Assume that
\begin{equation}\label{eq:threshold-characterization}
        \{m>0\}=\{\underline H>\ell\}
        \quad\text{up to null sets}.
\end{equation}
Then
\[
        |\ResGap|=0.
\]
\end{lemma}

\begin{proof}
Set
\[
        O:=\{x\in\Td:\ \underline H(x)<\ell(x)\}.
\]
Since \(\underline H\) and \(\ell\) are continuous, \(O\) is open.  By
\eqref{eq:threshold-characterization}, we have \(m=0\) a.e. in \(O\), which yields 
\(O\subset\Ovac\).
Hence we have
\[
        \Res\cap O=\emptyset.
\]
Since
\[
        \ResGap=\Res\cap\{\ell>\underline H\}=\Res\cap O,
\]
we conclude that \(|\ResGap|=0\).
\end{proof}
If \(|\ResGap|=0\), then Assumption \ref{ass:residual} imposes no additional
restriction.
Together with Theorem \ref{thm:main-unique}, this gives the following
convergence criterion.

\begin{corollary}
\label{cor:no-gap-residual-full}
Suppose that Assumptions \ref{ass:H}, \ref{ass:f}, and
\ref{ass:discounted-compactness} hold.  If
\[
        |\ResGap|=0,
\]
then \(\M_J\) contains at most one element up to additive constants.
Consequently, the normalized discounted value functions converge along the
whole family:
\[
        \langle u^\varepsilon\rangle\rightharpoonup u_*
        \quad\text{weakly in }W^{1,s}(\Td),
\]
where \(u_*\) is the unique selected minimizer normalized by
\[
        \int_{\Td}u_*\dx=0.
\]
\end{corollary}

\subsection{One-dimensional problems}
\label{sec:oned}

We first consider the one-dimensional case.  For background on
one-dimensional stationary mean field games with local coupling, where low
regularity, zero-density regions, and non-uniqueness phenomena can be seen
explicitly, we refer to \cite{GomesNurbekyanPrazeres2018}.

When \(d=1\), the ergodic system \eqref{eq:ergodic-intro} becomes
\begin{equation}\label{eq:1d-ergodic}
\begin{cases}
H(x,u_x)=f(x,m)+\lambda &\qquad\text{in }\Tone, \\[1mm]
-\bigl(mH_p(x,u_x)\bigr)_x=0 &\qquad\text{in }\Tone, \\[1mm]
m\ge0,\qquad \displaystyle\int_{\Tone}m\dx=1.
\end{cases}
\end{equation}

\begin{proposition}
\label{prop:1d-no-gap}
Let \(d=1\).  Suppose that Assumptions \ref{ass:H} and \ref{ass:f} hold.
Then
\[
        |\ResGap|=0.
\]
\end{proposition}

\begin{proof}
Fix \(u\in\E(m,\lambda)\).  From the continuity equation in
\eqref{eq:1d-ergodic}, there exists a constant \(j\in\R\) such that
\begin{equation}\label{eq:1d-flux}
        m(x)H_p(x,u_x(x))=j
        \quad\text{for a.e. }x\in\Tone.
\end{equation}
If \(j\ne0\), then \(m>0\) a.e. on \(\Tone\).  Hence the zero-density region
is null, and in particular
\[
        |\ResGap|=0.
\]

We may therefore assume \(j=0\).  Define
\begin{equation}\label{eq:p0-h0-1d}
        p_0(x):=\argmin_{p\in\R}H(x,p),
        \qquad
        h_0(x):=H(x,p_0(x)).
\end{equation}
By the strict convexity and coercivity of \(H(x,\cdot)\), \(p_0(x)\) is
uniquely defined.  Moreover,
\[
        h_0(x)=\underline H(x)
        \quad\text{for all }x\in\Tone.
\]

We prove that
\begin{equation}\label{eq:1d-threshold}
        \{m>0\}=\{h_0>\ell\}
        \quad\text{up to null sets}.
\end{equation}
On \(\{m>0\}\), \eqref{eq:1d-flux} with $j=0$ gives
\[
        H_p(x,u_x)=0
        \quad\text{a.e. on }\{m>0\}.
\]
By strict convexity,
\[
        u_x=p_0(x)
        \quad\text{a.e. on }\{m>0\}.
\]
Using the Hamilton--Jacobi equality on \(\{m>0\}\), we obtain
\[
        h_0(x)
        =
        H(x,p_0(x))
        =
        f(x,m(x))+\lambda
        >
        f(x,0)+\lambda
        =
        \ell(x)
        \quad\text{a.e. on }\{m>0\}.
\]
Thus
\[
        \{m>0\}\subset\{h_0>\ell\}
        \quad\text{up to null sets}.
\]

Conversely, on \(\{m=0\}\), the Hamilton--Jacobi inequality gives
\[
        H(x,u_x)\le \lambda+f(x,0)=\ell(x)
        \quad\text{a.e. on }\{m=0\}.
\]
Since
\[
        h_0(x)=\min_{p\in\R}H(x,p)\le H(x,u_x),
\]
we get
\[
        h_0(x)\le \ell(x)
        \quad\text{a.e. on }\{m=0\}.
\]
Hence
\[
        \{h_0>\ell\}\subset\{m>0\}
        \quad\text{up to null sets}.
\]
This proves \eqref{eq:1d-threshold}.  Since \(h_0=\underline H\), we have
\[
        \{m>0\}=\{\underline H>\ell\}
        \quad\text{up to null sets}.
\]
The assertion follows from Lemma \ref{lem:threshold-no-gap}.
\end{proof}

The preceding proposition gives the first concrete convergence result.

\begin{corollary}[Convergence in one dimension]
\label{cor:1d-convergence}
Let \(d=1\).  Suppose that Assumptions \ref{ass:H}, \ref{ass:f}, and
\ref{ass:discounted-compactness} hold.  Then the normalized discounted value
functions converge along the whole family:
\[
        \langle u^\varepsilon\rangle\rightharpoonup u_*
        \quad\text{weakly in }W^{1,s}(\Tone),
\]
where \(u_*\) is the unique selected minimizer normalized by
\[
        \int_{\Tone}u_*\dx=0.
\]
\end{corollary}

\begin{proof}
By Proposition \ref{prop:1d-no-gap}, we have \(|\ResGap|=0\).  The conclusion
follows from Corollary \ref{cor:no-gap-residual-full}.
\end{proof}

\subsection{Hamiltonians minimized at \texorpdfstring{\(p=0\)}{p=0}}
\label{sec:zero-minimizer}
We next give a multidimensional mechanism ensuring that the gap residual set
is null.  We assume that the pointwise minimizer of the Hamiltonian in the
second variable is independent of \(x\) and equal to the zero vector:
\begin{equation}\label{eq:zero-minimizer-condition}
        H(x,0)=\min_{p\in\R^d}H(x,p)
        \qquad\text{for every }x\in\Td .
\end{equation}
Under Assumption \ref{ass:H}, this minimizer is unique.

\begin{remark}
The condition \eqref{eq:zero-minimizer-condition} holds, for example, for
Hamiltonians of the form \(H(x,p)=h(p)+V(x)\) with
\(h(0)=\min_{p\in\R^d}h(p)=0\), and for quadratic Hamiltonians
\(H(x,p)=\langle A(x)p,p\rangle+V(x)\) with \(A(x)\) symmetric and uniformly
positive definite.
\end{remark}

\begin{proposition}
\label{prop:zero-minimizer-no-gap}
Suppose that Assumptions \ref{ass:H} and \ref{ass:f} hold, and assume
\eqref{eq:zero-minimizer-condition}.  Then
\[
        |\ResGap|=0.
\]
\end{proposition}

\begin{proof}
Let \(u\in\E(m,\lambda)\).  We prove the threshold characterization
\[
        \{m>0\}=\{\underline H>\ell\}
        \quad\text{up to null sets}.
\]
The conclusion then follows from Lemma \ref{lem:threshold-no-gap}.

Testing the continuity equation by \(u\), we obtain
\[
        \int_{\Td}mD_pH(x,Du)\cdot Du\dx=0.
\]
Since \(p=0\) is the minimizer of \(p\mapsto H(x,p)\), we have
\[
        D_pH(x,0)=0 \qquad\text{for every }x\in\Td .
\]
By the convexity of \(H(x,\cdot)\),
\[
        \bigl(D_pH(x,Du)-D_pH(x,0)\bigr)\cdot Du\ge0
        \quad\text{a.e. on }\Td .
\]
Hence
\[
        D_pH(x,Du)\cdot Du\ge0
        \quad\text{a.e. on }\Td .
\]
Since the integral of this nonnegative quantity with weight \(m\) is zero,
it vanishes \(m\,dx\)-a.e.  Equivalently, it vanishes Lebesgue-a.e. on
\(\{m>0\}\).  By the strict convexity of \(H(x,\cdot)\), we obtain
\[
        Du=0
        \quad\text{a.e. on }\{m>0\}.
\]

On \(\{m>0\}\), the Hamilton--Jacobi equality gives
\[
        \underline H(x)
        =
        H(x,0)
        =
        H(x,Du)
        =
        f(x,m(x))+\lambda
        >
        f(x,0)+\lambda =\ell(x).
\]
Thus
\[
        \{m>0\}\subset\{\underline H>\ell\}
        \quad\text{up to null sets}.
\]

Conversely, on \(\{m=0\}\), the Hamilton--Jacobi inequality gives
\[
        H(x,Du)\le \lambda+f(x,0)=\ell(x)
        \quad\text{a.e. on }\{m=0\}.
\]
Since \(\underline H(x)\le H(x,Du)\), we get
\[
        \underline H(x)\le \ell(x)
        \quad\text{a.e. on }\{m=0\}.
\]
Therefore
\[
        \{\underline H>\ell\}\subset\{m>0\}
        \quad\text{up to null sets}.
\]
We conclude that
\[
        \{m>0\}=\{\underline H>\ell\}
        \quad\text{up to null sets}.
\]
By Lemma \ref{lem:threshold-no-gap}, this implies that \(\ResGap\) is null.
\end{proof}

The second concrete convergence result follows.

\begin{corollary}[Convergence for Hamiltonians minimized at \(p=0\)]
\label{cor:zero-minimizer-convergence}
Suppose that Assumptions \ref{ass:H}, \ref{ass:f}, and
\ref{ass:discounted-compactness} hold, and assume
\[
        H(x,0)=\min_{p\in\R^d}H(x,p)
        \qquad\text{for every }x\in\Td .
\]
Then the normalized discounted value functions converge along the whole
family:
\[
        \langle u^\varepsilon\rangle\rightharpoonup u_*
        \quad\text{weakly in }W^{1,s}(\Td),
\]
where \(u_*\) is the unique selected minimizer normalized by
\[
        \int_{\Td}u_*\dx=0.
\]
\end{corollary}

\begin{proof}
By Proposition \ref{prop:zero-minimizer-no-gap}, we have \(|\ResGap|=0\).
The conclusion follows from Corollary \ref{cor:no-gap-residual-full}.
\end{proof}

\begin{remark}
In the convergence results above, the gap residual condition is verified
through the stronger property \(|\ResGap|=0\).  The case where
\(\ResGap\) has positive measure is not addressed here.
\end{remark}

\appendix
\section{A selection condition for vanishing discount limits}\label{sec:selection}

In this appendix, we recall the selection condition for vanishing discount limits.  The calculation follows the argument in \cite{GomesMitakeTerai2020,MitakeTerai2023}, but we give it for completeness.

\begin{proposition}[Selection condition]\label{prop:selection}
Suppose that Assumptions \ref{ass:H}, \ref{ass:f} and
\ref{ass:discounted-compactness} hold.  Let
\((u^{\varepsilon_n},m^{\varepsilon_n})\) be discounted weak solutions and let
\(\bar u\) be a subsequential limit as in Assumption
\ref{ass:discounted-compactness}.  Then
\begin{equation*}
        J(\bar u)\le J(v)
        \quad\text{for all }v\in\E(m,\lambda).
\end{equation*}
In particular, \(\bar u\in\M_J\).
\end{proposition}

\begin{proof}
Fix \(v\in\E(m,\lambda)\).  We omit the subscript \(n\) and write \(\varepsilon\), \(u^\varepsilon\), and \(m^\varepsilon\).  On \(\{m^\varepsilon>0\}\), the discounted Hamilton--Jacobi equation and the ergodic inequality for \(v\) give
\begin{equation*}
        f(x,m)+\lambda-f(x,m^\varepsilon)
        \ge H(x,Dv)-H(x,Du^\varepsilon)-\varepsilon u^\varepsilon.
\end{equation*}
By convexity of \(H(x,\cdot)\),
\begin{equation*}
        H(x,Dv)-H(x,Du^\varepsilon)
        \ge D_pH(x,Du^\varepsilon)\cdot D(v-u^\varepsilon).
\end{equation*}
Multiplying by \(m^\varepsilon\) and integrating, and then using the discounted continuity equation \eqref{eq:discounted} tested by \(v-u^\varepsilon\), we obtain
\begin{equation}\label{eq:sel-ineq1}
\begin{aligned}
        \varepsilon\int_{\Td}(v-u^\varepsilon)\dx
        -\varepsilon\int_{\Td}v m^\varepsilon\dx
        \le
        \int_{\Td}(f(x,m)-f(x,m^\varepsilon))m^\varepsilon\dx+\lambda.
\end{aligned}
\end{equation}
Indeed,
\[
\int_{\Td}m^\varepsilon D_pH(x,Du^\varepsilon)\cdot D(v-u^\varepsilon)\dx
=
\varepsilon\int_{\Td}(1-m^\varepsilon)(v-u^\varepsilon)\dx.
\]
Next, on \(\{m>0\}\), the ergodic Hamilton--Jacobi equation for \(v\) and the discounted inequality for \(u^\varepsilon\) yield
\begin{equation*}
        \varepsilon u^\varepsilon+H(x,Du^\varepsilon)-H(x,Dv)
        \le f(x,m^\varepsilon)-f(x,m)-\lambda.
\end{equation*}
Using convexity at \(Dv\),
\[
        H(x,Du^\varepsilon)-H(x,Dv)
        \ge D_pH(x,Dv)\cdot D(u^\varepsilon-v),
\]
multiplying by \(m\), and testing the ergodic continuity equation for \(v\) by \(u^\varepsilon-v\), we get
\begin{equation}\label{eq:sel-ineq2}
        \varepsilon\int_{\Td}u^\varepsilon m\dx
        \le
        \int_{\Td}(f(x,m^\varepsilon)-f(x,m))m\dx-\lambda.
\end{equation}
Adding \eqref{eq:sel-ineq1} and \eqref{eq:sel-ineq2}, the constants \(\lambda\) cancel and we find
\begin{equation*}
\begin{aligned}
        \varepsilon\int_{\Td}u^\varepsilon m\dx
        +\varepsilon\int_{\Td}(v-u^\varepsilon)\dx
        -\varepsilon\int_{\Td}v m^\varepsilon\dx
        \le
        \int_{\Td}(f(x,m^\varepsilon)-f(x,m))(m-m^\varepsilon)\dx.
\end{aligned}
\end{equation*}
By the monotonicity of \(f\) in \(m\), the right-hand side is nonpositive.  Dividing by \(\varepsilon>0\), we obtain
\begin{equation*}
        \int_{\Td}u^\varepsilon m\dx-\int_{\Td}u^\varepsilon\dx
        \le
        \int_{\Td}v m^\varepsilon\dx-\int_{\Td}v\dx.
\end{equation*}
The left-hand side is \(\int \langle u^\varepsilon\rangle m\dx\).  Passing to the limit, we get
\[
        \int_{\Td}\bar u m\dx
        \le
        \int_{\Td}v m\dx-\int_{\Td}v\dx.
\]
Since \(\bar u\) is normalized by construction, \(\int\bar u\dx=0\), this is precisely
\[
        J(\bar u)\le J(v).
\]
As \(v\in\E(m,\lambda)\) was arbitrary, \(\bar u\in\M_J\).
\end{proof}

\end{document}